\documentclass[graybox]{svmult}

\usepackage{mathptmx}       
\usepackage{helvet}         
\usepackage{courier}        
\usepackage{type1cm}        
\usepackage{makeidx}         
\usepackage{graphicx}        
\usepackage{multicol}        
\usepackage[bottom]{footmisc}

\usepackage{amsmath}
\usepackage{amssymb}
\usepackage{amscd}
\usepackage{indentfirst}

\def\1{{\bf 1}}

\makeindex             

\begin{document}

\title*{Order estimates of best orthogonal trigonometric approximations of classes of infinitely differentiable functions}
\titlerunning{Order estimates of best orthogonal trigonometric approximations}
\author{Tetiana A. Stepanyuk}

\authorrunning{T.~A.~Stepanyuk }

  
  \institute{Institute of Analysis and Number
  Theory Kopernikusgasse 24/II 8010, Graz, Austria, Graz University of Technology 
  \at Institute of Mathematics of Ukrainian National Academy of Sciences, 3, Tereshchenkivska st., 01601, Kyiv-4, Ukraine \\ \email{tania$_{-}$stepaniuk@ukr.net} }

\maketitle

\abstract{
In this paper we establish exact order estimates for the best uniform  orthogonal trigonometric approximations of the classes of $2\pi$-periodic functions, whose $(\psi,\beta)$--derivatives belong to unit balls of spaces $L_{p}$, $1\leq p<\infty$, in the case, when the sequence $\psi(k)$ tends to zero faster, than any power function, but slower than geometric progression. Similar estimates are also established in the $L_{s}$-metric, $1<s\leq\infty$ for the classes of differentiable functions, which $(\psi,\beta)$--derivatives belong to unit ball of space $L_{1}$.
}


\section{Introduction}


Let $L_{p}$,
$1\leq p<\infty$, be the space of $2\pi$--periodic functions $f$ summable to the power $p$
on  $[0,2\pi)$, with the norm
$\|f\|_{p}=\Big(\int\limits_{0}^{2\pi}|f(t)|^{p}dt\Big)^{\frac{1}{p}}$; $L_{\infty}$ be the space of $2\pi$--periodic functions  $f$, which are Lebesque measurable and essentially bounded    with the norm
$\|f\|_{\infty}=\mathop{\rm{ess}\sup}\limits_{t}|f(t)|$.

 Let $f:\mathbb{R}\rightarrow\mathbb{R}$ be the function from  $L_{1}$, whose Fourier series is given by
\begin{align*}
\sum_{k=-\infty}^{\infty}\hat{f}(k)e^{ikx},
\end{align*}
where  $\hat{f}(k)=\frac{1}{2\pi}\int\limits_{-\pi}^{\pi}f(t)e^{-ikt}dt$ are the Fourier coefficients of the function  $f$,
 $\psi(k)$ is an arbitrary fixed sequence of real numbers and $\beta$
is a fixed real number. Then, if the series
\begin{align*}
\sum_{k\in \mathbb{Z}\setminus\{0\}}\frac{\hat{f}(k)}{\psi(|k|)}e^{i(kx+\frac{\beta\pi}{2} \mathrm{sign} k)}
\end{align*}
\noindent is the Fourier series of some function $\varphi$ from $L_{1}$, then this function is called the $(\psi,\beta)$--derivative of the function $f$  and is denoted by
$f_{\beta}^{\psi}$.
A set of functions $f$, whose $(\psi,\beta)$--derivatives exist, is denoted by
 $L_{\beta}^{\psi}$  (see \cite{Stepanets1}).
 
Let
$$
B_{p}^{0}:=\left\{\varphi\in L_{p}: \ ||\varphi||_{p}\leq 1,
\ \varphi\perp1\right\}, \ \ 1\leq p\leq\infty.
$$
If $f\in L^{\psi}_{\beta}$, and, at the same time $f^{\psi}_{\beta}\in
B_{p}^{0}$, then we say that the function $f$ belongs to the class
  $L^{\psi}_{\beta,p}$.
 

By ${\mathfrak M}$ we denote the set of all convex (downward) continuous functions $\psi(t)$, $t\geq 1,$ such that $\lim\limits_{t\rightarrow\infty}\psi(t)=0$. Assume that the sequence $\psi (k),\ k\in \mathbb{N}$, specifying the class $L^{\psi}_{\beta,p}, \ 1\leq p\leq \infty$,  is the restriction of the functions  $\psi(t)$ from ${\mathfrak M}$
to the set of natural numbers.
 
 Following Stepanets (see, e.g., \cite{Stepanets1}), by using the characteristic  $\mu(\psi;t)$ of functions $\psi$ from $\in{\mathfrak M}$ of the form
\begin{equation}\label{mu}
\mu(t)=\mu(\psi;t):=\frac{t}{\eta(t)-t},
\end{equation}
where $\eta(t)=\eta(\psi;t):=\psi^{-1}\left(\psi(t)/2\right)$, $\psi^{-1}$ is the function inverse to $\psi$, we select the following subsets of the set   ${\mathfrak M}$:
$$
\mathfrak{M}^{+}_{\infty}=\left\{\psi\in \mathfrak{M}: \ \ \
\mu(\psi;t)\uparrow\infty \right\}.
$$
$$
\mathfrak{M}^{''}_{\infty}=\left\{\psi\in \mathfrak{M}^{+}_{\infty}: \ \ \ \exists K>0 \ \ 
\eta(\psi;t)-t\geq K \ \ \ t\geq 1 \right\}.
$$

 
The functions $\psi_{r,\alpha}(t)=\exp(-\alpha t^{r})$ are typical representatives of the set 
$\mathfrak{M}^{+}_{\infty}$. Moreover, if $r\in(0,1]$, 
then $\psi_{r,\alpha}\in \mathfrak{M}^{''}_{\infty}$. The classes $L^{\psi}_{\beta,p}$, generated by the functions $\psi=\psi_{r,\alpha}$ are denoted by $L^{\alpha,r}_{\beta,p}$.
 
 For functions  $f$ from classes $L^{\psi}_{\beta,p}$ we consider: $L_{s}$--norms of deviations of the functions  $f$ from their partial Fourier sums of order $n-1$, i.e., the quantities
\begin{equation}\label{rho}
\|\rho_{n}(f;\cdot)\|_{s}=
\|f(\cdot)-S_{n-1}(f;\cdot)\|_{s}, \ \ 1\leq s\leq\infty,
\end{equation}
where
$$
S_{n-1}(f;x)=\sum\limits_{k=-n+1}^{n-1}\hat{f}(k)e^{ikx};
$$
and the best orthogonal trigonometric approximations of the functions $f$ in metric of space  $L_{s}$, i.e., the quantities of the form
\begin{equation}\label{nn_term}
 e^{\bot}_{m}(f)_{s}=
  \inf\limits_{\gamma_{m}}\|f(\cdot)-S_{\gamma_{m}}(f;\cdot) \|_{s},  \ \  \ 1\leq s\leq\infty,
\end{equation}
where $\gamma_{m}$, $m\in\mathbb{N}$, is an arbitrary collection of  $m$ integer numbers, and
$$
S_{\gamma_{m}}(f;x)=\sum\limits_{k\in \gamma_{m}}\hat{f}(k)e^{ikx}.
$$

We set
\begin{equation}\label{Fsum}
 {\cal E}_{n}(L^{\psi}_{\beta,p})_{s}=
\sup\limits_{f\in
L^{\psi}_{\beta,p}}\|\rho_{n}(f;\cdot)\|_{s}, \ \ 1\leq p,s\leq\infty,
\end{equation}
\begin{equation}\label{n_term}
  e^{\bot}_{n}(L_{\beta,p}^{\psi})_{s}=\sup\limits_{f\in L_{\beta,p}^{\psi}}e^{\bot}_{n}(f)_{s}, \ 1\leq p,s\leq\infty.
\end{equation}

The following inequalities follow from given above definitions  (\ref{Fsum}) and (\ref{n_term})
\begin{equation}\label{ineq_comp}
  e^{\bot}_{2n}(L_{\beta,p}^{\psi})_{s}\leq 
  e^{\bot}_{2n-1}(L_{\beta,p}^{\psi})_{s}\leq {\cal E}_{n}(L^{\psi}_{\beta,p})_{s}, \ 1\leq p,s\leq\infty.
\end{equation}

In the case when $\psi(k)=k^{-r}$, $r>0$, the classes $L^{\psi}_{\beta,p}$, 
$1\leq p\leq\infty$, $\beta\in \mathbb{R}$ are well-known Weyl--Nagy classes $W^{r}_{\beta,p}$.
For these classes, the order estimates of quantities $ e^{\bot}_{n}(L_{\beta,p}^{\psi})_{s}$ are known for $1<p,s<\infty$ (see \cite{Romanyuk2002}, \cite{Romanyuk2007}), for $1\leq p<\infty$, $s=\infty$, $r>\frac{1}{p}$ and also for $p=1$, $1<s<\infty$, $r>\frac{1}{s'}$, $\frac{1}{s}+\frac{1}{s'}=1$ (see \cite{Romanyuk2007}, \cite{Romanyuk2012}).

In the case, when $\psi(k)$ tends to zero not faster than some power function, order estimates for quantities (\ref{n_term}) were established in \cite{Fedorenko1999}, \cite{S_S2015}, \cite{Shkapa2014_no2} and \cite{Shkapa2014_no3}. In the case, when $\psi(k)$ tends to zero not slower than geometric progression, exact order estimates for $ e^{\bot}_{n}(L_{\beta,p}^{\psi})_{s}$ were found in \cite{S_S_Dopovidi2015} for all $1\leq p,s\leq\infty$.

Our aim is to establish the exact-order estimates of $e^{\bot}_{n}(L_{\beta,p}^{\psi})_{\infty}$, $1\leq p<\infty$, and $e^{\bot}_{n}(L_{\beta,1}^{\psi})_{s}$, $1<s<\infty$, in the case, when $\psi$ decreases faster than any power function, but slower than geometric progression ($\psi\in\mathfrak{M}^{''}_{\infty}$).

\section{Best orthogonal trigonometric approximations of the classes $L^{\psi}_{\beta,p}$, $1<p<\infty$, in the uniform metric}

We write
 $a_{n}\asymp b_{n}$ to mean that there exist positive constants $C_{1}$ and
 $C_{2}$ independent of $n$ such that $C_{1}a_{n}\leq b_{n}\leq C_{2}a_{n}$ for
 all $n$.

\begin{theorem}\label{theoremUniform} 
Let ${1< p<\infty}$, $\psi\in\mathfrak{M}^{''}_{\infty}$
and the function $\frac{\psi(t)}{|\psi'(t)|}\uparrow\infty$  as  $t\rightarrow\infty$. Then, for all  $\beta\in \mathbb{R}$  the following order estimates hold
 \begin{equation}\label{theorem_1}
 e^{\bot}_{2n-1}(L_{\beta,p}^{\psi})_{\infty}\asymp e^{\bot}_{2n}(L_{\beta,p}^{\psi})_{\infty}\asymp \psi(n)(\eta(n)-n)^{\frac{1}{p}}.
\end{equation}
\end{theorem}

\begin{proof}
According to Theorem 1 from   \cite{S_S} under conditions $\psi\in\mathfrak{M}^{+}_{\infty}$,
$\beta\in \mathbb{R}$, ${1\leq p<\infty}$, for  $n\in
 \mathbb{N}$, such that ${\eta(n)-n\geq a>2}, \ {\mu(n)\geq b>2}$ the following estimate is true
\begin{equation}\label{th1}
  {\cal E}_{n}(L^{\psi}_{\beta,p})_{\infty}
\leq
K_{a,b} \ (2p)^{1-\frac{1}{p}} \psi(n)(\eta(n)-n)^{\frac{1}{p}},
\end{equation}
where
\begin{align*}
K_{a,b}=
\frac{1}{\pi}\max\left\{\frac{2b}{b-2}+\frac{1}{a}, \
2\pi\right\}.
\end{align*}

Using inequalities (\ref{ineq_comp}) and (\ref{th1}), we obtain
\begin{equation}\label{a1}
e^{\bot}_{2n}(L_{\beta,p}^{\psi})_{\infty}\leq e^{\bot}_{2n-1}(L_{\beta,p}^{\psi})_{\infty}\leq
K_{a,b} \ (2p)^{1-\frac{1}{p}} \psi(n)(\eta(n)-n)^{\frac{1}{p}}.
\end{equation}
Let us find the lower estimate for the quantity $e^{\bot}_{2n}(L_{\beta,p}^{\psi})_{\infty}$.
With this purpose we construct the function
\begin{align*}
 f^{*}_{p,n}(t)= f^{*}_{p,n}(\psi;t):=
\frac{\lambda_{p}}{\psi(n)(\eta(n)-n)^{\frac{1}{p'}}}\bigg(\frac{1}{2}\psi(1)\psi(2n)+
\end{align*}
\begin{equation}\label{function}
 +\sum\limits_{k=1}^{n-1}\psi(k)\psi(2n-k)\cos kt
+\sum\limits_{k=n}^{2n}\psi^{2}(k)\cos kt\bigg), \ \ \frac{1}{p}+\frac{1}{p'}=1.
\end{equation}

Let us show that $f^{*}_{p,n}\in L_{\beta,p}^{\psi}$. The definition of $(\psi,\beta)$--derivative yields
\begin{multline}\label{a2}
 (f^{*}_{p,n}(t))^{\psi}_{\beta}=
\frac{\lambda_{p}}{\psi(n)(\eta(n)-n)^{\frac{1}{p'}}}\bigg(\sum\limits_{k=1}^{n-1}\psi(2n-k)\cos \Big(kt+\frac{\beta\pi}{2}\Big)
 \\
 +\sum\limits_{k=n}^{2n}\psi(k)\cos \Big(kt+\frac{\beta\pi}{2}\Big)\bigg).
\end{multline}
Obviously
\begin{align*}
\big|(f^{*}_{p,n}(t))^{\psi}_{\beta}\big|\leq
\frac{\lambda_{p}}{\psi(n)(\eta(n)-n)^{\frac{1}{p'}}}\bigg(\sum\limits_{k=1}^{n-1}\psi(2n-k)+ \sum\limits_{k=n}^{2n}\psi(k)\bigg)<
\end{align*}
\begin{equation}\label{a3}
<\frac{2\lambda_{p}}{\psi(n)(\eta(n)-n)^{\frac{1}{p'}}}\sum\limits_{k=n}^{2n}\psi(k)
\leq\frac{2\lambda_{p}}{\psi(n)(\eta(n)-n)^{\frac{1}{p'}}}\bigg(\psi(n)+\int\limits_{n}^{\infty}\psi(u)du\bigg).
\end{equation}

To estimate the integral from the right part of formula
(\ref{a3}),
we use the following statement  \cite[p. 500]{Serdyuk2004}.

\begin{proposition}\label{statement1} If 
 $\psi\in\mathfrak{M}^{+}_{\infty}$, then for arbitrary ${m\in\mathbb{N}}$,
 such that $\mu(\psi,m)>2$
 the following condition holds
 \begin{equation}\label{prop1}
\int\limits_{m}^{\infty}\psi(u)du\leq
\frac{2}{1-\frac{2}{\mu(m)}}\psi(m)(\eta(m)-m).
\end{equation}
\end{proposition}

Formulas (\ref{a3}) and (\ref{prop1}) imply that
\begin{align*}
\big|(f^{*}_{p,n}(t))^{\psi}_{\beta}\big|\leq
\frac{2\lambda_{p}}{\psi(n)(\eta(n)-n)^{\frac{1}{p'}}}\bigg(\psi(n)+\frac{2b}{b-2}\psi(n)(\eta(n)-n)\bigg)<
\end{align*}
\begin{equation}\label{a7}
 <\frac{5\lambda_{p} b}{b-2}(\eta(n)-n)^{\frac{1}{p}}.
\end{equation}

We denote
\begin{equation}\label{Dkb}
D_{k,\beta}(t):=\frac{1}{2}\cos\frac{\beta\pi}{2}+\sum\limits_{j=1}^{k}\cos\Big(jt+\frac{\beta\pi}{2}\Big).
\end{equation}

Applying Abel transform, we have 
\begin{multline}\label{a4}
\sum\limits_{k=1}^{n-1}\psi(2n-k)\cos \Big(kt+\frac{\beta\pi}{2}\Big)=\sum\limits_{k=1}^{n-2}(\psi(2n-k+1)-\psi(2n-k))D_{k,\beta}(t)
\\
  +\psi(n+1)D_{n-1,\beta}(t)-\psi(2n-1)\frac{1}{2}\cos \frac{\beta\pi}{2}
\end{multline}
and
\begin{multline}\label{a5}
\sum\limits_{k=n}^{2n}\psi(k)\cos \Big(kt+\frac{\beta\pi}{2}\Big)=\sum\limits_{k=n}^{2n-1}(\psi(k)-\psi(k+1))D_{k,\beta}(t)
\\
  +\psi(2n)D_{2n,\beta}(t)-\psi(n)D_{n-1,\beta}(t).
\end{multline}

Since
\begin{equation}\label{grad}
\sum\limits_{k=0}^{N-1}\sin(\gamma+kt)=\sin\Big(\gamma+\frac{N-1}{2}t\Big)\sin
\frac{Ny}{2}\frac{1}{\sin\frac{t}{2}}
\end{equation}
 (see, e.g., \cite[p.43]{Gradshteyn}), for $N=k+1$, $\gamma=(\beta-1)\frac{\pi}{2}$, the following inequality holds
\begin{align}
& |D_{k,\beta}(t)|
=
\left| \frac{\cos\big(\frac{kt}{2}+
\frac{\beta\pi}{2}\big)\sin\frac{k+1}{2}t}{\sin\frac{t}{2}}-\frac{1}{2}\cos\frac{\beta\pi}{2}\right| \notag
\\
& =\left|\frac{\sin\big((k+\frac{1}{2})t+
\frac{\beta\pi}{2}\big)-\cos\frac{t}{2}
\sin\frac{\beta\pi}{2}}{2\sin\frac{t}{2}}\right| \leq \frac{\pi}{t}, \ \ \ 0<|t|\leq \pi.\label{ner}
\end{align}

According to (\ref{a2}), (\ref{a4}), (\ref{a5}) and (\ref{ner}), we obtain
\begin{align}
& \big|(f^{*}_{p,n}(t))^{\psi}_{\beta}\big|\leq
\frac{\lambda_{p}}{\psi(n)(\eta(n)-n)^{\frac{1}{p'}}}\frac{\pi}{|t|}\bigg(\sum\limits_{k=1}^{n-2}|\psi(2n-k)-\psi(2n-k-1)|+\psi(n+1) \notag
\\
&+\psi(2n-1)+\sum\limits_{k=n}^{2n-1}|\psi(k)-\psi(k+1)|+\psi(2n)+\psi(n)\bigg) \notag
\\
&  =
\frac{\lambda_{p}}{\psi(n)(\eta(n)-n)^{\frac{1}{p'}}}\frac{2\pi}{|t|}(\psi(n+1)+\psi(n))\leq\frac{4\pi\lambda_{p}}{(\eta(n)-n)^{\frac{1}{p'}}}\frac{1}{|t|}.\label{a6}
\end{align}

So,  (\ref{a7}) and (\ref{a6}) imply
\begin{align*}
 & \big\|(f^{*}_{p,n}(t))^{\psi}_{\beta}\big\|_{p} \notag
\\
& \leq
\lambda_{p}\max\Big\{\frac{5b}{b-2}, \ 4\pi\Big\}\bigg( \int\limits_{|t|\leq\frac{1}{\eta(n)-n}
}(\eta(n)-n)dt+\frac{1}{(\eta(n)-n)^{\frac{p}{p'}}}\int\limits_{\frac{1}{\eta(n)-n}\leq|t|\leq\pi}\frac{dt}{|t|^{p}}\bigg)^{\frac{1}{p}} \notag
\\
& \leq
2\lambda_{p}\max\Big\{\frac{5b}{b-2}, \ 4\pi\Big\}\Big( 1+\frac{1}{p-1}\Big)^{\frac{1}{p}}=2\lambda_{p}\max\Big\{\frac{5b}{b-2}, \ 4\pi\Big\}(p')^{\frac{1}{p}}. \notag
\end{align*}
Hence, for
\begin{align*}
\lambda_{p} =\frac{1}{2(p')^{\frac{1}{p}}\max\Big\{\frac{5b}{b-2}, \ 4\pi\Big\}}
\end{align*}
the embedding $f^{*}_{p,n}\in L^{\psi}_{\beta,p}$ is true.

Let us consider the quantity
\begin{equation}\label{eq3}
  I_{1}:=\inf\limits_{\gamma_{2n}}\bigg|\int\limits_{-\pi}^{\pi}(f^{*}_{p,n}(t)-S_{\gamma_{2n}}(f^{*}_{p,n};t))V_{2n}(t)dt\bigg|,
\end{equation}
where  $V_{2n}$ are de la $\mathrm{Vall\acute{e}e}$-Poisson kernels of the form
\begin{equation}\label{val_pus2}
  V_{m}(t):=\frac{1}{2}+\sum\limits_{k=1}^{m}\cos kt+2\sum\limits_{k=m+1}^{2m-1}\Big(1-\frac{k}{2m}\Big)\cos
kt, \ m\in
\mathbb{N}.
\end{equation}

 Proposition A1.1 from  \cite{Korn} implies
\begin{equation}\label{a9}
 I_{1}\leq\inf\limits_{\gamma_{2n}}\|f^{*}_{p}(t)-S_{\gamma_{2n}}(f^{*}_{p,n};t)\|_{\infty}\|V_{2n}\|_{1}=
 e^{\bot}_{2n}(f^{*}_{p,n})_{\infty}\|V_{2n}\|_{1}.
\end{equation}
Since (see, e.g., \cite[p.247]{Stepaniuk2014})
\begin{equation}\label{eq61}
 \|V_{m}\|_{1}\leq3\pi, \ \ m\in\mathbb{N},
\end{equation}
 from (\ref{a9}) and (\ref{eq61}) we can write down the estimate
\begin{equation}\label{eq9}
 e^{\bot}_{2n}(f^{*}_{p,n})_{\infty} \geq \frac{1}{3\pi}I_{1}.
\end{equation}

Notice, that
\begin{align}
& f^{*}_{p,n}(t)-S_{\gamma_{2n}}(f^{*}_{p,n};t) \notag
 \\
&=
 \frac{\lambda_{p}}{2\psi(n)(\eta(n)-n)^{\frac{1}{p'}}}
 \bigg(
{\mathop{\sum}\limits_{
 |k|\leq n-1,\atop k \notin\gamma_{2n} }}
\psi(|k|)\psi(2n-|k|)e^{ikt}+{\mathop{\sum}\limits_{
 n\leq|k|\leq 2n,\atop k \notin\gamma_{2n}}} \psi^{2}(|k|)e^{ikt}\bigg), \label{n1}
\end{align}
where $\psi(0):=\psi(1)$

Whereas
 \begin{equation}\label{int_riv}
  \int\limits_{-\pi}^{\pi}e^{ikt}e^{imt}dt=
{\left\{\begin{array}{cc}
0, \ & k+m\neq 0, \\
2\pi, &
k+m=0, \
  \end{array} \right.}  \ \ k,m\in\mathbb{Z},
\end{equation}
and taking into account  (\ref{val_pus2}), we obtain
\begin{equation}
\int\limits_{-\pi}^{\pi}(f^{*}_{p,n}(t)-S_{\gamma_{2n}}(f^{*}_{p,n};t))V_{2n}(t)dt
\end{equation} 
\begin{align}
&=\frac{\lambda_{p}}{4\psi(n)(\eta(n)-n)^{\frac{1}{p'}}}\int\limits_{-\pi}^{\pi}
 \bigg(
{\mathop{\sum}\limits_{
 0\leq k\leq n-1,\atop k \notin\gamma_{2n} }}
\psi(k)\psi(2n-k)e^{ikt}+{\mathop{\sum}\limits_{
 -n+1\leq k\leq -1,\atop k \notin\gamma_{2n} }}
\psi(|k|)\psi(2n-|k|)e^{ikt} \notag
\\
&+{\mathop{\sum}\limits_{
 n\leq k\leq 2n,\atop k \notin\gamma_{2n}}} \psi^{2}(k)e^{ikt}+
 {\mathop{\sum}\limits_{
-2n\leq k\leq -n,\atop k \notin\gamma_{2n}}} \psi^{2}(|k|)e^{ikt}\bigg)\times \notag
\\
& \times\Big(\sum\limits_{0\leq k\leq 2n}e^{ikt}+\sum\limits_{-2n\leq k\leq -1}e^{ikt}+
2\sum\limits_{2n+1\leq| k|\leq 4n-1}\Big(1-\frac{|k|}{4n}\Big)e^{ikt} \Big)dt 
\\
 & =\frac{\lambda_{p}\pi}{2\psi(n)(\eta(n)-n)^{\frac{1}{p'}}} \bigg(
{\mathop{\sum}\limits_{
 |k|\leq n-1,\atop k \notin\gamma_{2n} }}
\psi(|k|)\psi(2n-|k|)+
 {\mathop{\sum}\limits_{
n\leq |k|\leq 2n,\atop k \notin\gamma_{2n}}} \psi^{2}(|k|)\bigg).\label{a10}
\end{align}

The function $\phi_{n}(t):=\psi(t)\psi(2n-t)$ decreases for  ${t\in[1, n]}$. Indeed
\begin{align*}
\phi'_{n}(t)=|\psi'(t)||\psi'(2n-t)|\Big(\frac{\psi(t)}{|\psi'(t)|} -\frac{\psi(2n-t)}{|\psi'(2n-t)|}\Big)\leq0,
\end{align*}
because  $\frac{\psi(t)}{|\psi'(t)|}\uparrow\infty$ for large  $n$.

Thus, the monotonicity of function $\phi_{n}(t)$ and  (\ref{a10}) imply
\begin{align}
 &I_{1}
=\frac{\pi\lambda_{p}}{2\psi(n)(\eta(n)-n)^{\frac{1}{p'}}}\bigg( \psi^{2}(n)+
 {\mathop{\sum}\limits_{
n+1\leq |k|\leq 2n}} \psi^{2}(|k|)\bigg) \notag
\\
&>\frac{\pi\lambda_{p}}{2\psi(n)(\eta(n)-n)^{\frac{1}{p'}}}
\sum\limits_{k=n}^{2n} \psi^{2}(k)
\geq\frac{\pi\lambda_{p}}{2\psi(n)(\eta(n)-n)^{\frac{1}{p'}}}
\int\limits_{n}^{\eta(n)} \psi^{2}(t)dt \notag
\\
&>\frac{\pi\lambda_{p}}{2\psi(n)(\eta(n)-n)^{\frac{1}{p'}}}\psi^{2}(\eta(n))(\eta(n)-n)=\frac{\pi\lambda_{p}}{8}\psi(n)(\eta(n)-n)^{\frac{1}{p}}.\label{a11}
\end{align}

By considering (\ref{eq9}) and (\ref{a11}) we can write
\begin{equation}\label{h1}
 e^{\bot}_{2n}(L^{\psi}_{\beta,p})_{\infty}\geq e^{\bot}_{2n}(f^{*}_{p,n})_{\infty} \geq \frac{1}{3\pi}I_{1}\geq\frac{\lambda_{p}}{24}\psi(n)(\eta(n)-n)^{\frac{1}{p}}.
\end{equation}
Theorem 1 is proved.

\end{proof}

\begin{remark}\label{remarkTh1}
 Let $\psi\in\mathfrak{M}^{+}_{\infty}$,
$\beta\in \mathbb{R}$, ${1< p<\infty}$, $\frac{1}{p}+\frac{1}{p'}=1$, and the function $\frac{\psi(t)}{|\psi'(t)|}\uparrow\infty$ for $t\rightarrow\infty$. Then for $n\in
 \mathbb{N}$ the following estimates hold 
   \begin{equation}\label{remark1}
K_{b,p} \psi(n)(\eta(n)-n)^{\frac{1}{p}} \leq e^{\bot}_{2n}(L_{\beta,p}^{\psi})_{\infty}\leq e^{\bot}_{2n-1}(L_{\beta,p}^{\psi})_{\infty}  \leq K_{a,b,p} \psi(n)(\eta(n)-n)^{\frac{1}{p}},
\end{equation}
where 
\begin{equation}\label{Kab}
K_{a,b,p}=
\frac{1}{\pi}\max\left\{\frac{2b}{b-2}+\frac{1}{a}, \
2\pi\right\}(2p)^{\frac{1}{p'}}.
\end{equation}
\begin{equation}\label{Kb}
K_{b,p}= \frac{1}{48\max\Big\{\frac{5b}{b-2}, \ 4\pi\Big\}(p')^{\frac{1}{p}}}.
\end{equation}
\end{remark}

\section{Best orthogonal trigonometric approximations of the classes $L^{\psi}_{\beta,1}$ in the uniform metric}

\begin{theorem}\label{theoremComm}
 Let  $\psi\in\mathfrak{M}^{+}_{\infty}$. Then for all $\beta\in \mathbb{R}$ order estimates are true 
  \begin{equation}\label{theorem_3}
  e^{\bot}_{2n-1}(L_{\beta,1}^{\psi})_{\infty}\asymp e^{\bot}_{2n}(L_{\beta,1}^{\psi})_{\infty}\asymp \psi(n)(\eta(n)-n).
\end{equation}
\end{theorem}

\begin{proof}
According to formula (48)  from   \cite{S_S} under conditions $\psi\in\mathfrak{M}$,  $\sum\limits_{k=1}^{\infty}\psi(k)<\infty$,
$\beta\in \mathbb{R}$,  for all $n\in\mathbb{N}$ the following estimate holds
\begin{equation}\label{t2}
 {\cal E}_{n}(L^{\psi}_{\beta,1})_{\infty}
\leq\frac{1}{\pi}\sum\limits_{k=n}^{\infty}\psi(k).
\end{equation}
Using Proposition 1, we have
\begin{multline}\label{t4}
  e^{\bot}_{2n}(L^{\psi}_{\beta,1})_{\infty}\leq e^{\bot}_{2n-1}(L^{\psi}_{\beta,1})_{\infty}
\leq{\cal E}_{n}(L^{\psi}_{\beta,1})_{\infty}
\leq\frac{1}{\pi}\sum\limits_{k=n}^{\infty}\psi(k)
\\
\leq\frac{1}{\pi}\bigg(\psi(n)+\int\limits_{n}^{\infty}\psi(u)du\bigg)\leq
\frac{\psi(n)}{\pi}\bigg(1+\frac{b}{b-2}(\eta(n)-n)\bigg).
\end{multline}

Let us find the lower estimate for the quantity $e^{\bot}_{2n}(L_{\beta,1}^{\psi})_{\infty}$.

We consider the quantity
\begin{equation}\label{eq25}
  I_{2}:=\inf\limits_{\gamma_{2n}}\bigg|\int\limits_{-\pi}^{\pi}(f^{*}_{2n}(t)-S_{\gamma_{2n}}(f^{*}_{2n};t))V_{2n}(t)dt\bigg|,
\end{equation}
where  $V_{m}$ are de la $\mathrm{Vall\acute{e}e}$-Poisson kernels of the form (\ref{val_pus2}), and
\begin{equation}\label{eq26}
f^{*}_{m}(t)=f^{*}_{m}(\psi;t):=\frac{1}{5\pi m}\Big(\frac{1}{2}\psi(1)+\sum\limits_{k=1}^{m}k\psi(k)\cos kt+
\sum\limits_{k=m+1}^{2m}(2m+1-k)\psi(k)\cos kt\Big).
\end{equation}
In \cite[p. 263--265]{Stepaniuk2014}  it was shown that $\|(f^{*}_{m})^{\psi}_{\beta}\|_{1}\leq1$, i.e., $f^{*}_{m}$  belongs to the class
 $L^{\psi}_{\beta,1}$ for all $m\in\mathbb{N}$.

Using Proposition A1.1 from  \cite{Korn} and inequality (\ref{eq61}), we have
\begin{equation}\label{eq27}
I_{2}\leq\inf\limits_{\gamma_{2n}}\|f^{*}_{2n}(t)-S_{\gamma_{2n}}(f^{*}_{2n};t)\|_{\infty}\|V_{2n}\|_{1}
\leq3\pi \ e^{\bot}_{2n}(f_{2n}^{*})_{\infty}.
\end{equation}

Assuming again $\psi(0):=\psi(1)$, from (\ref{val_pus2}) and (\ref{eq26}), we derive
\begin{align}
 &I_{2}=\frac{1}{20\pi n}\inf\limits_{\gamma_{2n}}\bigg|\int\limits_{-\pi}^{\pi}
 \Big({\mathop{\sum}\limits_{
|k|\leq 2n,\atop k \notin\gamma_{2n} }}
 |k|\psi(|k|)e^{ikt}
 + {\mathop{\sum}\limits_{2n+1\leq| k|\leq 4n,\atop k \notin\gamma_{2n} }}
 (4n+1-|k|)\psi(|k|)e^{ikt}\Big)\times \notag
\\
&\times\Big(\sum\limits_{|k|\leq 2n}e^{ikt}+
2\sum\limits_{2n+1\leq| |k\leq 4n-1}\Big(1-\frac{|k|}{4n}\Big)e^{ikt}\Big)dt\bigg| \notag
\\
&
=\frac{1}{10 n}\inf\limits_{\gamma_{2n}}\bigg({\mathop{\sum}\limits_{
| k|\leq 2n,\atop k \notin\gamma_{2n} }}
 |k|\psi(|k|) 
 +{\mathop{\sum}\limits_{2n+1\leq| k|\leq 4n,\atop k \notin\gamma_{2n} }}\Big(1-\frac{|k|}{4n}\Big)
 \psi(|k|)\bigg) \notag
 \\
&>\frac{1}{10 n}\inf\limits_{\gamma_{2n}}
{\mathop{\sum}\limits_{
 | k|\leq 2n,\atop k \notin\gamma_{2n} }}
 |k|\psi(|k|)=
 \frac{1}{10 n}\bigg(n\psi(n)+2{\mathop{\sum}\limits_{
 k=n+1}^{2n}} k\psi(k)\bigg) \notag
\\
 & >\frac{1}{10 }{\mathop{\sum}\limits_{
 k=n}^{2n}} \psi(k)>\frac{1}{10 }\int\limits_{n}^{\eta(n)} \psi(t)dt>\frac{1}{20}\psi(n)(\eta(n)-n).\label{a15}
\end{align}

Formulas (\ref{eq27}) and (\ref{a15}) imply
\begin{align*}
e^{\bot}_{2n}(L_{\beta,1}^{\psi})_{\infty}\geq\ e^{\bot}_{2n}(f_{2n}^{*})_{\infty}\geq\frac{1}{3\pi}I_{2}>\frac{1}{60\pi}\psi(n)(\eta(n)-n).
\end{align*}
Theorem~\ref{theoremComm} is proved.
\end{proof}

\begin{remark}\label{remarkComm}
 Let $\psi\in\mathfrak{M}^{+}_{\infty}$ and
$\beta\in \mathbb{R}$. Then for  $n\in
 \mathbb{N}$, such that  ${\mu(n)\geq b>2}$ the following estimate hold
   \begin{equation}\label{remark3}
\frac{1}{60\pi}\psi(n)(\eta(n)-n) \leq e^{\bot}_{2n}(L_{\beta,1}^{\psi})_{\infty}\leq e^{\bot}_{2n-1}(L_{\beta,1}^{\psi})_{\infty}  \leq \frac{1}{\pi}\Big(\frac{1}{b}+\frac{b}{b-2}\Big) \psi(n)(\eta(n)-n).
\end{equation}
\end{remark}

\begin{corollary}\label{cor1}
 Let $r\in(0,1)$, $\alpha>0$, $1\leq p<\infty$ and $\beta\in\mathbb{R}$. Then for all $n\in\mathbb{N}$ the following estimates are true
  \begin{equation}\label{corollary1}
 e^{\bot}_{n}(L_{\beta,p}^{\alpha,r})_{\infty}\asymp \exp(-\alpha n^{r})n^{\frac{1-r}{p}}.
\end{equation}
\end{corollary}

\section{Best orthogonal trigonometric approximations of the classes $L^{\psi}_{\beta,1}$ in the  metric of spaces $L_{s}$, $1<s<\infty$}

\begin{theorem}\label{theoremIntegral}
 Let  ${1< s<\infty}$, $\psi\in\mathfrak{M}^{''}_{\infty}$
and function  $\frac{\psi(t)}{|\psi'(t)|}\uparrow\infty$ as $t\rightarrow\infty$ . Then for all $\beta\in \mathbb{R}$  order estimates hold
  \begin{equation}\label{theorem_2}
  e^{\bot}_{2n-1}(L_{\beta,1}^{\psi})_{s}\asymp e^{\bot}_{2n}(L_{\beta,1}^{\psi})_{s}\asymp \psi(n)(\eta(n)-n)^{\frac{1}{s'}}, \ \ \frac{1}{s}+\frac{1}{s'}=1.
\end{equation}
\end{theorem}

\begin{proof}
According to Theorem 2 from   \cite{S_S} under conditions $\psi\in\mathfrak{M}^{+}_{\infty}$,
$\beta\in \mathbb{R}$, ${1< s\leq\infty}$ for  $n\in
 \mathbb{N}$, such that ${\eta(n)-n\geq a>2}, \ {\mu(n)\geq b>2}$ the following estimate holds
\begin{equation}\label{th2}
  {\cal E}_{n}(L^{\psi}_{\beta,1})_{s}
\leq
K_{a,b} \ (2s')^{\frac{1}{s}} \psi(n)(\eta(n)-n)^{\frac{1}{s'}}.
\end{equation}

Using inequalities (\ref{ineq_comp}) and (\ref{th2}),  we get
\begin{equation}\label{a1}
e^{\bot}_{2n}(L_{\beta,1}^{\psi})_{s}\leq e^{\bot}_{2n-1}(L_{\beta,1}^{\psi})_{s}\leq
K_{a,b,s'} \ \left(2s'\right)^{\frac{1}{s}} \psi(n)(\eta(n)-n)^{\frac{1}{s'}}.
\end{equation}
Let us find the lower estimate of the quantity $e^{\bot}_{2n}(L_{\beta,1}^{\psi})_{s}$.

We consider the quantity
\begin{equation}\label{eq3}
  I_{3}:=\inf\limits_{\gamma_{2n}}\bigg|\int\limits_{-\pi}^{\pi}(f^{**}_{2n}(t)-S_{\gamma_{2n}}(f^{**}_{2n};t))f^{*}_{s',n}(t)dt\bigg|,
\end{equation}
where
\begin{align*}
f^{**}_{m}(t)=\frac{1}{3\pi}V_{m}(t),
\end{align*}
and $f^{*}_{s',n}$ is defined by formula (\ref{function}).

On the basis of  Proposition A1.1 from  \cite{Korn} we derive
\begin{equation}\label{for200}
 I_{3}\leq\inf\limits_{\gamma_{2n}}\|f^{**}_{2n}(t)-S_{\gamma_{2n}}(f^{**}_{2n};t)\|_{s}\|f^{*}_{s'}\|_{s'}\leq
 e^{\bot}_{2n}(f^{**}_{2n})_{s}.
\end{equation}

On other hand, using formulas (\ref{int_riv}), we write
\begin{align*}
& I_{3}=\frac{\lambda_{s'}}{12\pi\psi(n)(\eta(n)-n)^{\frac{1}{s}}}\inf\limits_{\gamma_{2n}}\bigg|\int\limits_{-\pi}^{\pi}
 \Big({\mathop{\sum}\limits_{
|k|\leq 2n,\atop k \notin\gamma_{2n} }}
e^{ikt}+
2{\mathop{\sum}\limits_{
2n+1\leq |k|\leq 4n-1,\atop k \notin\gamma_{2n} }}
\Big(1-\frac{|k|}{4n}\Big)e^{ikt}
\Big)\times \notag
\\
&\times \bigg(
{\mathop{\sum}\limits_{
| k|\leq n-1 }}
\psi(|k|)\psi(2n-|k|)e^{ikt}+
{\mathop{\sum}\limits_{
 n\leq |k|\leq 2n}} \psi^{2}(|k|)e^{ikt}
\bigg)dt\bigg| \notag
\end{align*}
\begin{align}
&=\frac{\lambda_{s'}}{6\psi(n)(\eta(n)-n)^{\frac{1}{s}}}\inf\limits_{\gamma_{2n}}
 \Big(
 {\mathop{\sum}\limits_{
 | k|\leq n-1,\atop k \notin\gamma_{2n} }}
\psi(|k|)\psi(2n-|k|)
+
{\mathop{\sum}\limits_{
 n\leq |k|\leq 2n,\atop k \notin\gamma_{2n} }} \psi^{2}(|k|)
\Big) \notag
 \\
&=\frac{\lambda_{s'}}{6\psi(n)(\eta(n)-n)^{\frac{1}{s}}}
 \Big(\psi^{2}(n)+
2\sum\limits_{k=n+1}^{2n} \psi^{2}(k)\Big)>
\frac{\lambda}{6\pi\psi(n)(\eta(n)-n)^{\frac{1}{s}}}
\sum\limits_{k=n}^{2n} \psi^{2}(k) \notag
\\
 & >
\frac{\lambda_{s'}}{6\psi(n)(\eta(n)-n)^{\frac{1}{s}}}\int\limits_{n}^{\eta(n)} \psi^{2}(t)dt>
\frac{\lambda_{s'}}{24}\psi(n)(\eta(n)-n)^{\frac{1}{s'}}.\label{a12}
\end{align}

Hence, formulas (\ref{for200}) and (\ref{a12}) imply
\begin{equation}\label{h2}
 e^{\bot}_{2n}(L^{\psi}_{\beta,1})_{s}\geq e^{\bot}_{2n}(f^{**}_{s'})_{s} \geq I_{3}\geq\frac{\lambda_{s'}}{24}\psi(n)(\eta(n)-n)^{\frac{1}{s'}}.
\end{equation}

Theorem~\ref{theoremIntegral} is proved.

\end{proof}

Note, that  functions 

1)  $e^{-\alpha t^{r}}t^{\gamma}$, $\alpha>0, \ r\in(0,1], \ \gamma\in\mathbb{R}$;

2)
$e^{-\alpha t^{r}}\ln (t+K)$, $\alpha>0, \ r\in(0,1], \ \ K>e-1$,

\noindent etc.,  can be regarded as examples of functions 
$\psi$, which satisfy the conditions of Theorem~\ref{theoremUniform} and Theorem~\ref{theoremIntegral}. 

\begin{remark}\label{remark2}
 Let $\psi\in\mathfrak{M}^{+}_{\infty}$,
$\beta\in \mathbb{R}$, ${1\leq p<\infty}$ and function $\frac{\psi(t)}{|\psi'(t)|}\uparrow\infty$  as $t\rightarrow\infty$. Then for all $n\in
 \mathbb{N}$, such tthe following estimates are true 
\begin{multline}
 K_{b,s'} \psi(n)(\eta(n)-n)^{\frac{1}{s'}} \leq e^{\bot}_{2n}(L_{\beta,1}^{\psi})_{s}\leq e^{\bot}_{2n-1}(L_{\beta,1}^{\psi})_{s}
\\
 \leq K_{a,b,s'} \psi(n)(\eta(n)-n)^{\frac{1}{s'}},
\end{multline}
where
$K_{a,b,s'}$ and $K_{b,s'}$ are defined by formulas (\ref{Kab}) and (\ref{Kb}) respectively.
\end{remark}

\begin{corollary}\label{cor2}
 Let $r\in(0,1)$, $\alpha>0$, $1<s<\infty$ and $\beta\in\mathbb{R}$. Then for all $n\in\mathbb{N}$ the following estimates are true
  \begin{equation}\label{corollary2}
 e^{\bot}_{n}(L_{\beta,1}^{\alpha,r})_{s}\asymp \exp(-\alpha n^{r})n^{\frac{1-r}{s'}}, \ \ \frac{1}{s}+\frac{1}{s'}=1.
\end{equation}
\end{corollary}

\section*{Acknowledgements} 
The author is supported by the Austrian Science Fund FWF
  project F5503 (part of the Special Research Program (SFB) 
``Quasi-Monte Carlo Methods: Theory and Applications'')


\vspace{10mm}
\begin{thebibliography}{99}%
%

%
\bibitem{Fedorenko1999}
A.~S.~Fedorenko, \textit{On the best m-term trigonometric and orthogonal trigonometric approximations of functions from the
classes $L^{\psi}_{\beta,p}$},    Ukr. Math. J., {\bf 51}:12 (1999), 1945--1949.





\bibitem{Gradshteyn} I.~S.~Gradshtein, I.~M.~Ryzhik,
 \textit{Tables of Integrals, Sums, Series, and Products} [in Russian], Fizmatgiz, Moscow (1963).

\bibitem{Korn}
N.~P.~Korneichuk, \textit{ Exact Constants in Approximation Theory}, {\bf 38}, Cambridge Univ. Press, Cambridge, New York (1990).

\bibitem{Romanyuk2002}
A.~S.~Romanyuk, \textit{Approximation of classes of periodic functions of many variables}, Mat. Zametki, {\bf 71}:1 (2002), 109–121 .









\bibitem{Romanyuk2007}
 A.~S.~Romanyuk, \textit{ Best trigonometric approximations of the classes of periodic functions of many variables in a uniform metric},
Mat. Zametki, {\bf 81}:2 (2007), 247–261 .

\bibitem{Romanyuk2012}
 A.~S.~Romanyuk, \textit{Approximate Characteristics of Classes of Periodic Functions of Many Variables} [in Russian], Institute of Mathematics, Ukrainian National Academy of Sciences, Kyiv (2012).





\bibitem{Serdyuk2004} A.~S.~Serdyuk, \textit{ Approximation by interpolation trigonometric polynomials on classes of periodic analytic functions}, Ukr. Mat. Zh., {\bf 64}:5  (2012), 698–712; English translation: Ukr. Math. J., {\bf 64}:5,  (2012), 797–815.


\bibitem{S_S} A.~S.~Serdyuk, T.~A.~Stepaniuk, \textit{Order estimates for the best approximation and approximation by Fourier sums of classes of infinitely differentiable functions}, Zb. Pr. Inst. Mat. NAN Ukr. {\bf 10}:1 (2013), 255-282. [in Ukrainian]


\bibitem{S_S2015}
A.~S.~Serdyuk, T.~A.~Stepaniuk, \textit{Order estimates for the best orthogonal
trigonometric approximations of the classes of convolutions of periodic functions of low
smoothness}, Ukr. Math. J., {\bf 67}:7 (2015),  1-24.

\bibitem{S_S_Dopovidi2015}
A.~S.~Serdyuk, T.~A.~Stepaniuk, \textit{Estimates of the best m-term trigonometric
approximations of classes of analytic functions}, Dopov. Nats. Akad. Nauk Ukr., Mat. Pryr. Tekh.
Nauky, No. 2 (2015), 32--37. [in Ukrainian]

\bibitem{Shkapa2014_no2}
V.~V.~Shkapa, \textit{Estimates of the best M-term and orthogonal trigonometric approximations of functions from the classes $L^{\psi}_{\beta,p}$ in
a uniform metric}, Differential Equations and Related Problems [in Ukrainian], Institute of Mathematics, Ukrainian National
Academy of Sciences, Kyiv, {\bf 11}:2 (2014), 305--317.


\bibitem{Shkapa2014_no3}
 V.~V.~Shkapa, \textit{Best orthogonal trigonometric approximations of functions from the classes  $L^{\psi}_{\beta,1}$}, Approximation Theory of
Functions and Related Problems [in Ukrainian], Institute of Mathematics, Ukrainian National Academy of Sciences, Kyiv, {\bf 11}:3
(2014), 315--329.



\bibitem{Stepanets1}
A.~I.~Stepanets,  \textit{ Methods of Approximation Theory}, VSP: Leiden, Boston (2005).



\bibitem{Stepaniuk2014} T.~A.~Stepaniuk, \textit{ Estimates of the best approximations and approximations of Fourier sums of classes of convolutions of periodic functions of not high smoothness in integral metrics}, Zb. Pr. Inst. Mat. NAN Ukr. {\bf 11}:3 (2014),  241-269. [in Ukrainian]



\end{thebibliography}
\end{document}